\documentclass[11pt]{article}
\usepackage{etex}

\usepackage{pgf,tikz}
\usepackage{mathrsfs}
\usepackage{graphicx}
\usepackage{caption}
\usepackage{subcaption}

\usetikzlibrary{arrows}

\usetikzlibrary{arrows}
\usepackage[margin=1in]{geometry}
\usepackage{stackengine}

\usepackage{color}
\usepackage[latin1]{inputenc}
\usepackage[T1]{fontenc}
\usepackage[normalem]{ulem}
\usepackage[english]{babel}
\usepackage{verbatim}
\usepackage{graphicx}
\usepackage{enumerate}
\usepackage{amsmath,amssymb,amsfonts,amsthm,amscd,mathrsfs}
\usepackage{array}
\usepackage{amsmath,amssymb,graphicx,stmaryrd,enumerate,bbm,alltt}
\usepackage{dsfont}
\usepackage{comment}

\usepackage[T1]{fontenc}
\usepackage{babel}
\usepackage[bookmarksopen, bookmarksnumbered]{hyperref}

\usepackage{url}
\usepackage{pgfplots}

\usepackage{caption}

\newtheorem{theorem}{Theorem}[section]

\newtheorem{lemma}[theorem]{Lemma}
\newtheorem{prop}[theorem]{Proposition}

\theoremstyle{definition}
\newtheorem{defn}[theorem]{Definition}

\input xy
\xyoption{all}

\DeclareMathOperator{\Bern}{Bern}

\newcommand{\E}{\mathbb E}
\newcommand{\PP}{\mathbb P}
\newcommand{\Z}{\mathbb Z}
\newcommand{\R}{\mathbb R}

\newcommand{\bS}{\overline S}
\newcommand{\bC}{\overline C}

\newcommand{\bv}{\mathbf v}

\newcommand{\boo}{\mathbf 0}

\newcommand{\cB}{\mathcal B}
\newcommand{\cC}{\mathcal C}
\newcommand{\cF}{\mathcal F}

\newcommand{\cK}{\mathcal K}
\newcommand{\cW}{\mathcal W}
\newcommand{\cM}{\mathcal M}
\newcommand{\cS}{\mathcal S}
\newcommand{\cT}{\mathcal T}

\newcommand{\cP}{\mathcal P}
\newcommand{\cQ}{\mathcal Q}
\newcommand{\cR}{\mathcal R}
\newcommand{\cI}{\mathcal I}
\newcommand{\cJ}{\mathcal J}

\newcommand{\hG}{\tilde G}

\newcommand{\hP}{\tilde P}
\newcommand{\hJ}{\tilde J}

\newcommand{\Wv}{W^{\vee}}

\newcommand{\fW}{\mathfrak W}

\newcommand{\sP}{\mathscr P}

\makeatletter
\renewcommand\tableofcontents{
  \null\hfill\textbf{\Large\contentsname}\hfill\null\par
  \@mkboth{\MakeUppercase\contentsname}{\MakeUppercase\contentsname}
  \@starttoc{toc}
}

\g@addto@macro\normalsize{
  \setlength\abovedisplayskip{5pt}
  \setlength\belowdisplayskip{5pt}
  \setlength\abovedisplayshortskip{3pt}
  \setlength\belowdisplayshortskip{3pt}
}

\makeatother

\numberwithin{equation}{section}

\begin{document}
\definecolor{ttqqqq}{rgb}{0.2,0.,0.}

\title{Stationary Distributions for the Voter Model in $d\geq 3$ are Factors of IID}

\author{Allan Sly
\thanks{Department of Mathematics, Princeton University, Princeton, NJ, USA. e-mail: allansly@princeton.edu}
\and Lingfu Zhang
\thanks{Department of Mathematics, Princeton University, Princeton, NJ, USA. e-mail: lingfuz@math.princeton.edu}
}
\date{}

\maketitle

\begin{abstract}
For the Voter Model on $\Z^d$, $d\geq 3$,
we show that the (extremal) stationary distributions are isomorphic to Bernoulli shifts, and answer an open question asked by Steif and Tykesson in \cite{steif2017generalized}.
The proof gives explicit constructions of the stationary distributions as factors of IID processes on $\Z^d$.
\end{abstract}

\section{Introduction}
We study the stationary distributions of the \emph{Voter Model} in $\Z^d$, for $d \geq 3$.
The model is one of the classic interacting particle systems taking values in $\{0,1\}^{\Z^d}$ and can be defined as follows.
On each vertex $x \in \Z^d$, there is a voter with an opinion $\eta(x) \in \{0, 1\}$.
Each voter at rate $1$ chooses a neighbor at random (among its $2d$ neighbors, with equal probability), and changes its opinion to the same as that neighbor.
This model can also be seen as a continuous time Markov process with state space $\{0, 1\}^{\Z^d}$.
For each measure $\mu$ on $\{0, 1\}^{\Z^d}$, and $t \in \R_+$, define $\cM_t\mu$ as the measure of running the Markov process for time $t$ with initial measure $\mu$.

For this Markov process, all the extremal stationary distributions can be described as follows (see e.g. \cite{liggett2004interacting}).
For each $0 \leq p \leq 1$, let $\rho_p$ be the measure on $\{0, 1\}^{\Z^d}$, where each voter has opinion $\Bern(p)$ independently.
Let $\mu_p$ be the weak limit of $\cM_t\rho_p$ as $t \rightarrow \infty$, then $\mu_p$ is a stationary distribution of the Voter Model.
For $d \leq 2$, where simple random walk is recurrent, only the constant measures $\mu_0, \mu_1$ are extremal and they are the only extremal stationary distributions.
For $d \geq 3$, the transient case, $\mu_p$ is extremal for each $0 \leq p \leq 1$, and $\{\mu_p\}_{p\in [0, 1]}$ are precisely all the extremal stationary distributions.

We consider the ergodic properties of the family $\{\mu_p\}_{p\in [0, 1]}$ for $d \geq 3$.
It is known that they are translation invariant and spatially ergodic.
A stronger ergodic property is the so called \emph{Bernoullicity}, defined as follows.
\begin{defn}  \label{defn:ber}
Let $X$ be any finite set equipped with a probability measure.
The product space $X^{\Z^d}$, together with
the product measure and an action of $\Z^d$ given by translations, is called a \emph{Bernoulli shift}.
It is called a \emph{generalized Bernoulli shift} if $X$ is replaced by a general probability space (i.e. not necessarily finite).
\end{defn}

The question of whether $\{\mu_p\}_{p\in [0, 1]}$ for $d \geq 3$ are isomorphic to Bernoulli shifts was posed by Steif and Tykesson in~\cite[Question 7.21]{steif2017generalized}. We give an affirmative answer to this question.
\begin{theorem}  \label{thm:main}
When $d \geq 3$, $\mu_p$ is isomorphic to a Bernoulli shift for each $0 \leq p \leq 1$.
\end{theorem}

Steif and Tykesson were more generally interested in the question of what they called \emph{Generalized Divide and Color models}.  In such models the vertices of a graph are partitioned into subsets by a \emph{random equivalence relation (RER)} and then each equivalence class of vertices is given a random color independently.
Examples of this include the Ising and Potts models via the Random Cluster Model (RCM) and, as we will see, so is the Voter Model through its dual formulation.  When the partition is isomorphic to a Bernoulli shift and its elements are finite almost surely then it is easy to see that the resulting coloring is also isomorphic to a Bernoulli shift.  Steif and Tykesson asked whether there were natural examples where the equivalence classes of the partition are infinite but that the coloring process is nonetheless isomorphic to a Bernoulli shift.  The Voter Model provides such an example answering Question 7.23 of~\cite{steif2017generalized}.

To establish that $\mu_p$ is isomorphic to a (generalized) Bernoulli shift on $\Z^d$, by~\cite{ornstein1970factors} it is sufficient to show that it is a factor of IID. Since the measure theoretical entropy of $\mu_p$ is clearly finite, by the Ornstein Isomorphism Theorem~ \cite{ornstein1970bernoulli} we get Theorem~\ref{thm:main}.
Note that while the results in \cite{ornstein1970factors,ornstein1970bernoulli} are for $\Z$ actions, they are generalized to amenable groups (see e.g. \cite{ornstein1987entropy}).

It is not hard to see that for each $t$, $\cM_t\rho_p$ is a factor of an IID process.
Our approach will then be to find a sequence of times $t_1, t_2, \ldots$, and couple all of $\cM_{t_1}\rho_p, \cM_{t_2}\rho_p, \ldots$ together, such that the resulting coupling is also a factor of IID.
The coupling will be defined so that almost surely the configuration converges in $\{0,1\}^{\Z^d}$ (in the product topology).

To do so, we consider the dual process of the Voter Model, and interpret $\cM_t\rho_p$ as the color process of the random equivalent relations given by coalescing simple random walks, see e.g. \cite[Section 14.3]{aldous1995reversible} and \cite[Section 1.3.4]{steif2017generalized}.
The coalescing simple random walks can be described as following: at each vertex in $\Z^d$ we start a continuous time random walker with jump rate $1$, and any two walkers coalesce when they meet at the same vertex.
For each $t \in \R_+$, we define a random equivalent relation (RER) as following: for any $x, y \in \Z^d$, they are in the same class if the walkers starting from $x$ and $y$ coalesce before time $t$.
For every equivalent class (which we also denote as a cluster) we take a random $\Bern(p)$ variable independently, and we let $\eta_t(x)$ be the same as that random variable for each $x$ in the cluster.
Then $\{\eta_t(x) \}_{x\in \Z^d}$ is the color process of this RER, and is distributed as $\cM_t\rho_p$.

For each $t\in\R_+$, the RER is defined in the probability space of the coalescing simple random walks, so a naive way of coupling $\cM_t\rho_p$ for different $t$ is to simply color each cluster of each time independently.
However, in this way, for fixed $x \in \Z^d$ $\eta_t(x)$ is IID for all $t$, so $\cM_t\rho_p$ does not strongly converge as $t \rightarrow \infty$.

Our approach is to couple $\cM_t\rho_p$ inductively.
Specifically, we take a sequence of times $t_k = 2^k$ for integers $k\geq 0$.
For each $k\geq 0$, given the coalescing simple random walks up to time $2^k$, and a coloring of the remaining walkers, we construct random walks from these walkers to time $2^{k+1}$.
The constructed walks are not independent simple random walks any more, but favoring the event that walkers of the same color coalesce.  Nonetheless, when we average over the whole process it will still have the correct distribution.

In particular, the construction shall satisfy the following requirements.
First, it is a factor of IID.
Second, if the coloring of the remaining walkers at time $2^k$ are IID $\Bern(p)$, then the marginal distribution of the constructed walks is the same as that of coalescing simple random walks up to time $2^{k+1}$.  Finally, conditional on the coalescing walks up to time $2^{k+1}$, we recolor a small portion of the remaining walkers, in a sense we will make quantitative, so that the coloring of the remaining walkers are made IID $\Bern(p)$ again.
This construction produces a coupling between the colorings of the walkers at each $2^{k}$ such that almost surely each walker changes its color only finitely many times.  We thus get almost sure convergence of the coupled color processes, or $\cM_{2^k}\rho_p$, as $k \rightarrow \infty$ and the limit process is a factor of IID.

In the literature there has been interest in Bernoullicity (or proving factor of IID) for other models as well; see e.g. \cite{van1999existence, holroyd2016finitely, lyons2017factors, spinka2020finitary, spinka2020finitely, ray2020proper}.
Another natural class of divide and color models are the Ising and Potts models.  The RCM representation gives a partition of $\Z^d$ and the law of the Ising and Potts models is then constructed by assigning each component of the RCM an independent state.  Since any infinite component of the RCM will be unique almost surely on $\Z^d$, the resulting measure will be a factor of IID if and only if the components are all finite almost surely.  On the tree, however, there may be many infinite components of the RCM.  It remains an open problem of Lyons~\cite{lyons2017factors} to determine when the free Ising model on the infinite d-regular tree is a factor of IID.  It is known that it is a factor of IID at high temperatures when the model has uniqueness and is not a factor of IID at low temperatures when the measure is non-extremal.  However, an intermediate regime where the measure is non-unique but extremal remains open.  We conjecture that it is a factor of IID which would give an alternative example for Question 7.23~\cite{steif2017generalized}.

\section*{Acknowledgments}
The authors would like to thank Jeff Steif for suggesting the problem.  
We are also very thankful for the many helpful suggestions from an anonymous referee, who pointed out an error in the first version of this paper, and suggested a simplification in the construction of $\fW_0$ and $\fW_1$ in Section \ref{ssec:biased}.  This work was supported by an NSF career award (DMS-1352013), a Simons Investigator grant and a MacArthur Fellowship.

\section{Coupling of two times}  \label{sec:StepConst}
From now on we fix $d \geq 3$ and $p \in [0, 1]$.

In this section we describe the construction of the coupling at each step, between two times $t_0$ and $t_0 + t$.
We do it with the following approach:
given some walkers at time $t_0$, and a $0-1$ coloring of these walkers, we construct coalescing simple random walks for these walkers from $t_0$ to $t_0 + t$,
and we define a new coloring for the remaining walkers at time $t_0+t$.

We construct these coalescing simple random walks and the new coloring from the following IID process.
For each vertex $x\in\Z^d$ we take a sequence of simple random walks $\{W_{x,m}\}_{m=1}^{\infty}$, each for time $[0, t]$; a sequence of uniform $[0, 1]$ random variables $\{u_{x,m}\}_{m=1}^{\infty}$; and an additional uniform $[0,1]$ random variable $v_x$.
Here and below, by simple random walks we mean one starts at the origin $\boo$, unless otherwise noted.

Formally, we let $\cW_t$ be the subspace of all left continuous functions from $[0, t]$ to $\Z^d$, which take only finitely many values (we refer to such functions as \emph{paths}).
For a (finite or infinite) collection of paths, we say they are \emph{coalesce-able}, if each vertex in $\Z^d$ is visited by at most finitely many of these paths.
By \emph{coalescing paths}, we mean some coalesce-able paths such that if two of them are at the same vertex at a time, they remain the same at all later times.

Let's explain how coalescing paths will be constructed. From any collection of coalesce-able paths, we first index them by natural numbers, then construct the coalescing paths by joining them one at a time.
When the paths are simple random walks, such construction gives coalescing random walks.

We now set up some notations for this construction.
\begin{defn}
Let $\sP$ be a (possibly empty) collection of coalescing paths, and $P$ be another path, and all of them are in $\cW_t$.
We denote $J[P, \sP] \in \cW_t$ as the path of joining $P$ into $\sP$, 
\[
J[P, \sP] (t') :=
\begin{cases}
P(t'), \quad 0 \le t' \leq t_h, \\
\overline{P}(t'), \quad  t_h < t' \le t.
\end{cases}
\]
where $t_h := \inf\{t' \in [0, t]: P(t') \in \{P'(t')\}_{P'\in \sP}\}\cup \{t\}$, and when $t_h<t$, $\overline{P}$ is any path in $\sP$ satisfying  $\lim_{t'\downarrow t_h}\overline{P}(t') = \lim_{t'\downarrow t_h}P(t')$.
In words, $J[P, \sP]$ is the same as $P$ until $P$ hits a path in $\sP$ for the first time, and then it follows that path after the hitting.

A set $I$ is called an \emph{index set} if it is either $\emptyset$, or $\{1,\ldots, k\}$ for some $k \in \Z_+$, or $\Z_+$.
For a coalesce-able collection of paths $\{P_{i}\}_{i\in I} \in \cW_t^{I}$, 
we inductively define $J[\{P_{i}\}_{i\in I}, \sP]: I \times [0, t]\rightarrow \Z^d$, read as \emph{paths $\{P_i\}_{i\in I}$ coalesced into the paths $\sP$}, by joining the paths in $\{P_{i}\}_{i\in I}$ one by one; i.e.
we first let
\[
J[\{P_{i}\}_{i\in I}, \sP](1, \cdot):=
J[P_{1}, \sP],
\]
and for each $i \in I, i>1$, we let
\[
J[\{P_{i}\}_{i\in I}, \sP](i, \cdot):=
J[\{P_{i}\}_{i\in I}, \sP \cup \{ J[\{P_{i}\}_{i\in I}, \sP](i', \cdot) \}_{i'\in I, i'<i} ].
\]
We will write $J[\{P_{i}\}_{i\in I}]$ in the case that $\sP=\emptyset$ and we are only coalescing $\{P_{i}\}_{i\in I}$.
Note that $J[\{P_{i}\}_{i\in I}, \sP]\cup \sP$ and $J[\{P_{i}\}_{i\in I}]$ are now coalescing paths.
\end{defn}
Take an index set $I$, and take vertices $x_i\in\Z^d$ and independent simple random walks $W_i$ for each $i\in I$.
Then $J[\{W_{i}+x_i\}_{i\in I}]$ is distributed as coalescing simple random walks, and the law does not depend on the order of joining the paths (although the realization itself does depend on the order).
In fact, the recursive joining operation corresponds to revealing paths (in coalescing simple random walks) one after another.
For any given coalescing paths $\sP$,
the paths $J[\{W_{i}+x_i\}_{i\in I}, \sP]$ is distributed as coalescing simple random walks conditioned on a set of existing paths of some walkers.

\subsection{Construction by groups}
So far the construction of coalescing paths relies on an ordering of the paths.
When we work with simple random walks with infinitely many different starting vertices, there is no translation invariant way to index them by natural numbers. 
In this case we need the starting vertices to be \emph{sparse}, so that we can split the walks into finite sets. For each of them we construct a local ordering, and construct coalescing walks separately.
\begin{defn}
For any $t\in \R_+$, $S \subset \Z^d$, and paths $\{P_{x}\}_{x \in S}\in \cW_t^{S}$ indexed by $S$,
consider a graph $\mathcal{G}$ (called \emph{the graph of $\{P_{x}\}_{x \in S}$}) as follows.
The vertex set of $\mathcal{G}$ is $S$, and for any $x_1, x_2\in S$, there is an edge between them if $P_{x_1}$ and $P_{x_2}$ ever meet.
Note that when we construct coalescing paths from $\{P_{x}\}_{x \in S}$, the ordering only matters within connected components of $\mathcal{G}$.
The paths $\{P_{x}\}_{x \in S}$ are said to be a \emph{non-percolate family for $S$}, if $\mathcal{G}$ contains no infinite connected component.

When $S$ is infinite, it is said to be \emph{$t$-sparse}, if almost surely $\{W_{x}+x\}_{x \in S}$ is a non-percolate family for $S$, 
Here $W_x$ is an independent simple random walk of time $t$ for each $x\in S$.
\end{defn}
We show that low density sets are sparse.
\begin{prop}  \label{prop:NonPer}
Let $S \subset \Z^d$ be a random set from site percolation where each vertex is in $S$ independently with probability $\overline{p}$.
Then for any $t\in \R_+$, there exists $\delta(t)>0$ such that if $\overline{p}< \delta(t)$ then $S$ is $t$-sparse almost surely.
\end{prop}
\begin{proof}
Take an independent simple random walk $W_x$ on $[0, t]$ for each $x \in \Z^d$.
It suffices to consider the probability of the event, that $\boo \in S$, and the connected component of $\boo$ in the graph of $\{W_{x}+x\}_{x \in S}$ is finite.
Let $R \in \Z_+$, we consider the probability that $\boo$ is connected to some $x \in S$ having $\|x\|_1 > R$.
We show that this probability decays to zero as $R \rightarrow \infty$.
Actually, it is bounded by
\[
\sum_{\substack{\{x_0=\boo, x_1,\ldots, x_k\}\subset \Z^d, \\ k\geq 1, \|x_k\|_1>R}}
\PP[\boo \in S, \exists t_1, \ldots, t_k \in [0, t], W_{x_{i-1}}(t_i)+x_{i-1}=W_{x_i}(t_i)+x_i,
x_i \in S, \forall 1 \leq i \leq k].
\]
For each $x\in \Z^d$, let $R_x := \max_{t' \in [0, t]} \|W_x(t')\|_1$.
Then the above sum is bounded by
\begin{equation}  \label{eq:prop:NonPer2}
\sum_{k=1}^{\infty} \sum_{\substack{x_0=\boo, x_1,\ldots, x_k, \in \Z^d,\\ r_0, \ldots, r_k \in \Z_{+}, 2(r_0+\ldots +r_k)\geq R,\\ \|x_{i-1} - x_i\|_1 \leq r_{i-1}+r_i}}
\prod_{i=0}^k
\PP[R_{x_i} = r_i, x_i \in S].
\end{equation}
For each $x \in \Z^d$, and any $r \in \Z_+$, we have that $\PP[R_{x} = r, x \in S] = \overline{p}\PP[R_{\boo}=r ]$.
Thus \eqref{eq:prop:NonPer2} is further bounded by
\[
\sum_{k=1}^{\infty} \sum_{\substack{r_0, \ldots, r_k \in \Z_{+},\\ 2(r_0+\ldots +r_k)\geq R}}
\PP[R_{\boo} = r_0]\delta(t)
\prod_{i=1}^k 
\left(\PP[R_{\boo} = r_i]\delta(t)(2(r_{i-1}+r_i)+1)^d\right)
\leq
\sum_{k=1}^{\infty} \Upsilon_{k, R}
\]
where
\[
\Upsilon_{k, R}:=
\sum_{\substack{r_0, \ldots, r_k \in \Z_{+},\\ 2(r_0+\ldots +r_k)\geq R}}
\prod_{i=0}^k
\left((5r_i)^{2d}\PP[R_{\boo} = r_i]\delta(t)\right).
\]
We first consider the case where $R=1$. We have
\[
\Upsilon_{k, 1}=
\prod_{i=0}^k
\sum_{r_i \in \Z_{+}}
(5r_i)^{2d}\PP[R_{\boo} = r_i]\delta(t)
=
\left(\sum_{r\in \Z_{+}}
(5r)^{2d}\PP[R_{\boo} = r]\delta(t)\right)^{k+1}.
\]
By Lemma \ref{lem:NonPer} below, we know that
\begin{equation}  \label{eq:prop:NonPer5}
\sum_{r\in \Z_{+}}(5r)^{2d}\PP[R_{\boo} = r] < \infty .
\end{equation}
By taking $\delta(t)= (2\sum_{r\in \Z_{+}}(5r)^{2d}\PP[R_{\boo} = r])^{-1}$, we have $\Upsilon_{k, 1} = 2^{-k-1}$ and $\sum_{k=1}^{\infty} \Upsilon_{k, 1} < \infty$.

For $R\in \Z_+$, we note that $\Upsilon_{k, R}$ monotonically decays as $R$ increases.
Using that $\Upsilon_{k, R} \le \Upsilon_{k, 1} = 2^{-k-1}$,
to show that $\sum_{k=1}^{\infty} \Upsilon_{k, R}$ decays to zero as $R \rightarrow \infty$, it suffices to show that for each $k \in \Z_+$, $\lim_{R\rightarrow\infty} \Upsilon_{k, R} = 0$.
Indeed, we have that
\[
\begin{split}
\Upsilon_{k, R} & \leq
(k+1)
\sum_{r_0, \ldots, r_k \in \Z_{+}, r_0 \geq R/2(k+1)}
\prod_{i=0}^k
\left((5r_i)^{2d}\PP[R_{\boo} = r_i]\delta(t)\right)
\\
&=
(k+1)2^{-k}
\left(\sum_{r \geq R/2(k+1)}
(5r)^{2d}\PP[R_{\boo} = r]\delta(t)\right).
\end{split}
\]
By \eqref{eq:prop:NonPer5}, we have
\[
\lim_{R\rightarrow \infty}
\sum_{r \geq R/2(k+1)}
(5r)^{2d}\PP[R_{\boo} = r]\delta(t) = 0.
\]
This implies that \eqref{eq:prop:NonPer2} decays to zero as $R \rightarrow \infty$, and our conclusion follows.
\end{proof}
\begin{lemma}   \label{lem:NonPer}
Let $t \in \R_+$, $W$ be a simple random walk on $[0, t]$, and let $R_{\boo}=\max_{t' \in [0, t]} \|W(t')\|_1$.
Then for any $m \in \Z_+$, we have $\sum_{r=1}^{\infty} r^m\PP[R_{\boo} = r] < \infty$.
\end{lemma}
\begin{proof}
If $R_{\boo} = r$, then there must be at least $r$ jumps for $W$;
thus we have
$\PP[R_{\boo} = r] \leq \PP[K \geq r]$, where $K$ is a Poisson random variable with rate $t$.
Thus we get
\[
\sum_{r=1}^{\infty} r^m\PP[R_{\boo} = r]
\leq
\E\left[ \sum_{r=1}^K r^m \right]
\leq \E\left[ K^{m+1} \right] < \infty,
\]
so our conclusion follows.
\end{proof}

Now we give an alternative construction of coalescing simple random walks, for sets that are $t$-sparse.
We resolve each connected component of the graph of $\{W_{x}+x\}_{x \in S}$ individually, as by construction they will not affect each other.  This means that rather than indexing the entire infinite set $S$, it is enough to have an ordering on each of the finite components which can be done simply in a translation invariant way.

We set up notations for this by group construction of coalescing paths.
\begin{defn}\label{defn:sparseContruction}
Take $t \in \R_+$ and a collection of coalescing paths $\sP$.
Take $S \subset \Z^d$ with some total ordering $\prec$, and let $\{P_{x}\}_{x \in S}\in \cW_t^{S}$ be a non-percolate family. 

We define $\hJ[\{P_{x}\}_{x \in S}, \sP], \hJ[\{P_{x}\}_{x \in S}]: S\times [0, t] \rightarrow \Z^d$ as follows.
For any connected component of the graph of $\{P_{x}\}_{x \in S}$, denoted as $x_1 \prec \ldots \prec x_{k} \in S$,
we define
\[
\hJ[\{P_{x}\}_{x \in S}, \sP]( x_i,\cdot):=J[\{P_{x_j}\}_{j=1}^{k}, \sP](i,\cdot),
\]
and
\[
\hJ[\{P_{x}\}_{x \in S}](x_i, \cdot):= J[\{P_{x_j}\}_{j=1}^k](i, \cdot).
\]
\end{defn}
When $S$ is $t$-sparse, and each $P_x=W_x+x$ where $W_x$ are independent simple random walks, these constructions give coalescing simple random walks.

\subsection{Biased coupling of paths and coloring} \label{ssec:biased}
In this subsection we will give an explicit construction of the coupling of two times.

We are given walkers with initial locations $S \subset \Z^d$, and an initial random coloring $C:S\to\{0, 1\}$, such that $C(x)=1$ with probability $p$ independently for each $x\in S$.
Recall that the extra randomness is from the IID process, where for each vertex $x\in\Z^d$ we have a sequence of simple random walks $\{W_{x,m}\}_{m=1}^{\infty}$, and uniform $[0, 1]$ random variables $\{u_{x,m}\}_{m=1}^{\infty}$ and $v_x$.
We will construct coalescing simple random walks from each vertex in $S$, denoted by $\{P_x\}_{x\in S} \in \cW_t^{S}$; and a new coloring $\bC$ of the remaining walkers $\bS:=\{P_x(t):x\in S\}$.
Our construction has the following properties:
\begin{enumerate}
\item The walks $\{P_x\}_{x \in S}$ are distributed as coalescing random walks from $S$.
\item Conditional on $\{P_x\}_{x \in S}$, the coloring $\bC$ of the remaining walkers satisfies that $\bC(y)=1$ with probability $p$ independently for each $y \in \bS$.
\end{enumerate}

In order to resolve the problem of ordering, we split the vertices of $S$ into $M$ random groups, each of which will be almost surely $t$-sparse and construct the ordering in each group sequentially.
So the first step is to take $M:= \left\lceil\max \{\delta(t)^{-1}, t^2 \} \right\rceil$, where $\delta(t)$ is given by Proposition~\ref{prop:NonPer}.
The groups are given by the random variables $\{v_x\}_{x\in S}$ (which we also denote as $\bv$ below), by letting $G_l := \left\{x \in S: \lfloor v_x M \rfloor=l \right\}$ for $0 \leq l < M$. Then each $G_l$ is almost surely $t$-sparse by Proposition \ref{prop:NonPer}.
We construct the paths for walkers in each $G_l$ sequentially.  Finally to apply the construction from Definition~\ref{defn:sparseContruction} we will need a total ordering $\prec$ on $\Z^d$ which we define to be the dictionary order by coordinates which is translation invariant.

We start with the group $G_0$.
Almost surely, $\{W_{x,1}+x\}_{x \in G_0}$ is a non-percolate family for $G_0$, and we assume that this is indeed the case for our choice of $\{ (W_{x,1}, v_x ) \}_{x \in S}$.
Using the ordering $\prec$, we can coalesce these paths, by taking $P_{x}:=\hJ[\{W_{x',1}+x'\}_{x' \in G_0}](x,\cdot)$ for each $x \in G_0$.
Next we color the remaining walkers, represented by $\{P_x(t): x \in G_0\}$.
For each $y \in \{P_x(t): x \in G_0\}$, there are only finitely many $x \in G_0$ with $P_x(t)=y$; we denote $\varphi(y)$ to be the smallest one (in the ordering $\prec$), and we let $\bC(y):=C(\varphi(y))$.

For each $l \geq 1$, denote $\hG_l:=\bigcup_{i=0}^{l-1}G_i$ to be the union of the earlier groups.
Let $\sP_l:=\{P_x\}_{x\in\hG_l}$ be the constructed coalescing paths, and $Y_l:=\{P_x(t): x\in \hG_l \}$ be the remaining walkers of these groups.  Given $\bv$ (the random variables used to define the groups), the paths $\sP_l$ and a coloring $\bC(y)$ for each remaining walker $y \in Y_l$ we define the paths from $G_l$.

Fix $x \in G_l$.  We could just choose a random walk from $x$ and let it coalesce with the already defined paths $\sP_l$.  But this would lead to too large a probability that $C(x)\neq \bC(P_x(t))$. Instead we take a biased sample of the walk whose distribution will depend on $C(x)$.  Specifically, given $\sP_l$ and $\{\bC(y)\}_{y\in Y_l}$, we will construct two measures $\fW_0,\fW_1$, on the space of walks with a coloring, such that $(1-p)\fW_0 + p\fW_1$ is the law of a simple random walk, and its coloring would be the same as the coloring of the existing path in $\sP_l$ that it coalescing into; or it would be $\Bern(p)$ if the simple random walk does not coalesce into any path in $\sP_l$.  Then when we pick the walk for $x$ and the coloring according to $\fW_{C(x)}$, the law after averaging over $C(x)$ gives a simple random walk since $C(x)$ is $\Bern(p)$ independent of $\sP_l$ and $\{\bC(y)\}_{y\in Y_l}$, and the coloring would be as desired.

Recall that for each $x\in G_l$ we have a sequence of random walks $\{W_{x,m}\}_{m=1}^\infty$.
Let $y_{x,m} := J[W_{x,m} + x, \sP_l] (t)$ if the endpoint is in $Y_l$, and $y_{x,m} := \Xi$ otherwise, where $\Xi$ is just a notation.
Let $C'_{x,m}:=\bC(y_{x,m})$ if $y_{x,m}\neq \Xi$, and let $C'_{x,m}$ be random $\Bern(p)$ otherwise (below we will explicitly define $C'_{x,m}$ as a function of the big IID process).
We note that (as will be made precise below), in the later case $C'_{x,m}$ can depend on $C(x)$, while it is $\Bern(p)$ after averaging over $C(x)$ (conditional on $\bv$, $\sP_l$, and $\{\bC(y)\}_{y\in Y_l}$).
We wish to pick random $m^*(x)$, depending on $\bv$, $\sP_l$, $\{\bC(y)\}_{y\in Y_l}$ and $C(x)$, such that $C'_{x,m^*(x)} = C(x)$ with probability as close to 1 as possible, and after averaging over $C(x)$ (while still conditional on $\bv$, $\sP_l$, $\{\bC(y)\}_{y\in Y_l}$) the law of $W_{x,m^*(x)}$ is simply a simple random walk, and $C'_{x,m^*(x)}$ is $\Bern(p)$ or equals $\bC(y_{x,m^*(x)})$, depending on whether $y_{x,m^*(x)}= \Xi$ or not.

Now let's define $\fW_0, \fW_1$ and $m^*(x)$.
For $x\in G_l$, we define
\[
\cK_{x,0} :=
\PP\Big[  y_{x,m}\neq \Xi, \bC(y_{x,m}) = 0 \mid \bv, \sP_l, \{\bC(y)\}_{y\in Y_l}\Big],
\]
and
\[
\cK_{x,1} :=
\PP\Big[  y_{x,m}\neq \Xi, \bC(y_{x,m}) = 1 \mid \bv, \sP_l, \{\bC(y)\}_{y\in Y_l}\Big].
\]
We note that these conditional probabilities do not depend on $m$.
In words, $\cK_{x,0}$ (resp. $\cK_{x,1}$) is the probability that a simple random walk starting from $x$ coalesces into a vertex in $Y_l$ with color $0$ (resp. $1$), given $\bv$, $\sP_l$, and $\{\bC(y)\}_{y\in Y_l}$.
We also denote
\[
\cK_{x,\Xi} :=
\PP\Big[ y_{x,m}= \Xi \mid \bv, \sP_l\Big],
\]
which also does not depend on $m$, and is the probability that a simple random walk starting from $x$ does not hit any existing path, given $\bv$, $\sP_l$.
From these definitions we immediately have that $\cK_{x,0}+\cK_{x,1}+\cK_{x,\Xi}=1$.

We then define
\[
w_{x,m}:=
\begin{cases}
\cK_{x,0} u_{x,m}, & \text{if } y_{x,m}\neq \Xi, \bC(y_{x,m}) = 0, \\
\cK_{x,0} + \cK_{x,\Xi}u_{x,m}, & \text{if } y_{x,m} = \Xi,\\
\cK_{x,0} + \cK_{x,\Xi} + \cK_{x,1} u_{x,m}, & \text{if } y_{x,m}\neq \Xi, \bC(y_{x,m}) = 1 .
\end{cases}
\]
In words, given $\bv$, $\sP_l$, $\{\bC(y)\}_{y\in Y_l}$, and $W_{x,m}$, the distribution of $w_{x,m}$ is uniform in $A_{x,0}$, $A_{x,\Xi}$, or $A_{x,1}$, in these three cases respectively.
Here $A_{x,0}:=[0, \cK_{x,0})$, $A_{x,\Xi}:=[\cK_{x,0}, \cK_{x,0}+\cK_{x,\Xi})$, and $A_{x,1}:=[\cK_{x,0}+\cK_{x,\Xi}, 1]$ (see Figure \ref{fig:int} for an illustration).
Conditional on only $\bv$, $\sP_l$ and $\{\bC(y)\}_{y\in Y_l}$, the lengths of these three intervals are equal to $\cK_{x,0}$, $\cK_{x,\Xi}$, and $\cK_{x,1}$, respectively; so the distribution of $w_{x,m}$ is uniform in $[0,1]$.

We now define $C'_{x,m}$, such that $C'_{x,m}=1$ if $w_{x,m}>(1-p)\cK_{x,\Xi} + \cK_{x,0}$, and $C'_{x,m}=0$ otherwise.
Thus conditional on $\bv$, $\sP_l$, and $\{\bC(y)\}_{y\in Y_l}$, we have $C'_{x,m}=\bC(y_{x,m})$ if $y_{x,m}\neq \Xi$, and $C'_{x,m}$ is $\Bern(p)$ if $y_{x,m}=\Xi$, independently for all such $x\in G_l$ and $m\in\Z_+$.

\begin{figure}
\begin{tikzpicture}[line cap=round,line join=round,>=triangle 45,x=1.835cm,y=1.0cm]
\clip(-3,-0.5) rectangle (6,0.5);
\draw [line width=.8pt] (-10,0.0) -- (10,0.0);
\draw [line width=.8pt, color=red] (-2.5,0.0) -- (0.5,0.0);
\draw [line width=.8pt, color=blue] (1.4,0.0) -- (5.5,0.0);

\draw [fill=black] (0.5,0.0) circle (1.2pt);
\draw [fill=black] (1.4,0.0) circle (1.2pt);
\draw [fill=black] (-2.5,0.0) circle (1.2pt);
\draw [fill=black] (5.5,0.0) circle (1.2pt);

\draw (-2.5,0.0) node[anchor=north]{$0$};
\draw (5.5,0.0) node[anchor=north]{$1$};
\draw (0.95,0.0) node[anchor=south]{$A_{x,\Xi}$};
\draw (-1,0.0) node[anchor=south]{$A_{x,0}$};
\draw (3.45,0.0) node[anchor=south]{$A_{x,1}$};

\end{tikzpicture}
\caption{Divide the interval $[0, 1]$ into segments $A_{x,0}$, $A_{x,\Xi}$, and $A_{x,1}$, with lengths $\cK_{x,0}$, $\cK_{x,\Xi}$, and $\cK_{x,1}$, respectively.}
\label{fig:int}
\end{figure}

We let $\fW_0$ denote the measure of $(W_{x,m}, C'_{x,m})$ conditional on $w_{x,m}\leq 1-p$ and let $\fW_1$ denote the measure of $(W_{x,m}, C'_{x,m})$ conditional on $w_{x,m}\geq 1-p$ (and as always conditional on $\bv$, $\sP_l$, $\{\bC(y)\}_{y\in Y_l}$).
Note that this definition is independent of $m$, and from this definition we have that $(1-p)\fW_0+p\fW_1$ is the law of $(W_{x,m}, C'_{x,m})$ for each $m\in\Z_+$.

We choose $m^*(x)$, to be the first $m$ such that $w_{x,m}$ lies in $[0,1-p]$ if $C(x)=0$ or  the first $m$ such that $w_{x,m}$ lies in $[1-p,1]$ if $C(x)=1$, that is
\[
m^*(x):=
\begin{cases}
\min \{m\in \Z_+: w_{x,m}\leq 1-p \},\;\; \mathrm{if }\;C(x)=0, \\
\min \{m\in \Z_+: w_{x,m}\geq 1-p \},\;\; \mathrm{if }\;C(x)=1.
\end{cases}
\]
Since $C(x)$ is $\Bern(p)$, the distribution of $w_{x,m^*(x)}$ is still uniform in $[0, 1]$, conditional on $\bv$, $\sP_l$, and $\{\bC(y)\}_{y\in Y_l}$ but averaging over $C(x)$. We have the following result of the path $W_{x,m^*(x)}$ and coloring $C'_{x,m^*(x)}$.
\begin{lemma}  \label{lem:CondiRandWalk}
Conditional on $\bv$, $\sP_l$, and $\{\bC(y)\}_{y\in Y_l}$, we have that the pair $(W_{x,m^*(x)}, C'_{x,m^*(x)})$ has the same distribution as $(W_{x,m}, C'_{x,m})$ (for any $m\in\Z_+$), and are independent among all $x \in G_l$.
\end{lemma}
\begin{proof}
For each $x \in G_l$, given $m^*(x)$ and $C(x)$ (in addition to $\bv$, $\sP_l$, and $\{\bC(y)\}_{y\in Y_l}$), the law of $(W_{x,m^*(x)}, C'_{x,m^*(x)})$ is then given by $\fW_{C(x)}$.
As $C(x)=1$ and $C(x)=0$ have probability $p$ and $1-p$ respectively, we have that the law of $(W_{x,m^*(x)}, C'_{x,m^*(x)})$ is $(1-p)\fW_0+p\fW_1$, which is the same as $(W_{x,m}, C'_{x,m})$ for each $m\in\Z_+$. Finally, the independence follows from the conditional (on $\bv$, $\sP_l$, $\{\bC(y)\}_{y\in Y_l}$) independence of $C(x)$ and $W_{x,m}$, $u_{x,m}$, among all $x \in G_l$ and $m\in\Z_+$.
\end{proof}

We then construct coalescing paths on $G_l$, using paths $\{W_{x,m^*(x)}\}_{x \in G_l}$ and the same method as $G_0$.
By Lemma \ref{lem:CondiRandWalk}, $\{W_{x,m^*(x)}+x\}_{x \in G_l}$ is almost surely a non-percolate family for $G_l$.
Using the order $\prec$, we take $P_{x}:=\hJ[\{W_{x',m^*(x')} + x'\}_{x' \in G_l}, \sP_l](x, \cdot)$ for each $x \in G_l$.
For each $y \in Y_{l+1}\backslash Y_l$, there are only finitely many $x \in G_l$ with $P_x(t)=y$;
we denote $\varphi(y)\in S$ to be the smallest (in the ordering $\prec$) such that $P_{\varphi(y)}(t) = y$, and we let $\bC(y):=C'_{\varphi(y),m^*(\varphi(y))}$.

Finally, by sequentially processing each $G_l$, $l=0,1,\ldots, M-1$, we have constructed $P_x$ for each $x \in S$.
We let $\bS:=\{P_x(t):x\in S\}$, then we have also defined $\bC(y)$ for each $y \in \bS$.

Our construction satisfies the two properties stated at the beginning of this subsection.
\begin{lemma}  \label{lem:const-prop}
For each $0\le l <M$, conditional on $\bv$ the walks $\sP_{l+1}=\{P_x\}_{x\in\hG_{l+1}}$ are coalescing random walks; and conditional on $\bv$ and $\sP_{l+1}$, the coloring $\{\bC(y)\}_{y\in Y_{l+1}}$ are IID $\Bern(p)$.
Moreover, the walks $\{P_x\}_{x\in S}$ are coalescing random walks; and conditional on $\{P_x\}_{x\in S}$, the coloring $\{\bC(y)\}_{y\in \bS}$ are IID $\Bern(p)$.
\end{lemma}
\begin{proof}
The first statement follows by the fact that $\hJ$ produces coalescing random walks when the input is independent random walks which is satisfied by Lemma~\ref{lem:CondiRandWalk}.

For the second statement we do induction in $l$.
For group $G_0$, since $\varphi(y)$ for all $y \in Y_1=\{P_x(t):x\in G_0\}$ are mutually distinct in $G_0$ and $\{C(x)\}_{x\in G_0}$ are IID $\Bern(p)$, we have that (conditional on $\bv$ and $\sP_1$) $\{\bC(y)\}_{y\in Y_1}$ are IID $\Bern(p)$.

For $G_l$ where $l>0$, we assume that $\{\bC(y)\}_{y\in Y_l}$ are IID $\Bern(p)$ conditional on $\bv$ and $\sP_l$.
By this induction hypothesis, and noting that (by Lemma \ref{lem:CondiRandWalk}) $\{\bC(y)\}_{y\in Y_l}$ and $\{W_{x,m^*(x)}\}_{x \in G_l}$ are independent conditional on $\bv$, $\sP_l$, we have that $\{\bC(y)\}_{y\in Y_l}$ are IID $\Bern(p)$, conditional on $\bv$, $\sP_l$, $\{W_{x,m^*(x)}\}_{x \in G_l}$.
By Lemma~\ref{lem:CondiRandWalk} again, we have that conditional on $\bv$, $\sP_l$, $\{\bC(y)\}_{y\in Y_l}$ and $\{W_{x,m^*(x)}\}_{x \in G_l}$, the coloring $C'_{x,m^*(x)}$ are IID $\Bern(p)$ for all $x\in G_l$ with $y_{x,m^*(x)}=\Xi$.
Then since that $\varphi(y)$ for all $y \in \bS$ are mutually distinct in the set $\{x\in G_l: y_{x,m^*(x)}=\Xi\}$, we have that $\{\bC(y)\}_{y\in Y_{l+1}\setminus Y_l}$ are IID $\Bern(p)$, conditional on $\bv$, $\sP_l$, $\{W_{x,m^*(x)}\}_{x \in G_l}$, and $\{\bC(y)\}_{y\in Y_l}$.
Thus $\{\bC(y)\}_{y\in Y_{l+1}} = \{\bC(y)\}_{y\in Y_l}\cup \{\bC(y)\}_{y\in Y_{l+1}\setminus Y_l}$ are IID $\Bern(p)$ conditional on $\bv$ and $\sP_{l+1}$, as $\sP_{l+1}$ is determined by $\sP_l$ and $\{W_{x,m^*(x)}\}_{x \in G_l}$.
\end{proof}

We conclude with an analysis of the probability that a walker changes its color; i.e. $C(x)\neq \bC(P_x(t))$, for each $x \in S$.
The bound is given by considering the probabilities of random walks intersecting in different ways, encoded by the following definition.
\begin{defn}   \label{defn:CoalProb}
For $t\in\R_+$, and paths $P_1, P_1', \ldots, P_k, P_k' \in \cW_t$, let $T_i:=\inf\{t'\in [0, t]: P_i(t') = P_i'(t')\}\cup \{\infty\}$ for each $i = 1, \ldots, k$.
Define $\cI\left((P_1, P_1'), \ldots, (P_k, P_k')\right)$ as the characteristic function of the event $0<T_1<\ldots<T_k<\infty$.
\end{defn}
\begin{prop}  \label{prop:ConstrChangeProb}
For any given $S \subset \Z^d$ and each $x \in S$, we have
\begin{equation}  \label{eq:ConstrChangeProb:0}
\begin{split}
\PP[ C(x)\neq \bC(P_x(t)) ] &\leq 2t^{-1} +
\Bigg(\sum_{x_1\in S, x_1 \neq x}\frac{1}{2}\E[\cI((W_1+x_1, W_2+x),(W_1+x_1, W_3+x)) ]
\\
+ \sum_{\substack{x_1,x_2\in S, \\ x\neq x_1, x\neq x_2, x_1\neq x_2}}
&\frac{1}{2}\E[\cI((W_1+x_1, W_2+x_2),(W_1+x_1, W_3+x),(W_1+x_1, W_4+x)) ]
\\
+&\frac{1}{2}\E[\cI((W_1+x_1, W_3+x),(W_1+x_1, W_2+x_2),(W_1+x_1, W_4+x)) ]
\\
+&\frac{1}{2}\E[\cI((W_1+x_1, W_3+x),(W_2+x_2, W_4+x),(W_1+x_1, W_2+x_2)) ]\Bigg)^{1/2},
\end{split}
\end{equation}
where $W_1, W_2, W_3, W_4$ are independent simple random walks in $\cW_t$.
\end{prop}
To better illustrate the meaning of the upper bound, in Figure \ref{fig:coa} we draw the patterns of coalescing of simple random walks, 
corresponding to the last three lines of \eqref{eq:ConstrChangeProb:0}.
\begin{figure}
\centering
\begin{minipage}{.29\textwidth}
  \centering

\subcaption{$\cI\left(\substack{(W_1+x_1, W_2+x_2),\\(W_1+x_1, W_3+x),\\(W_1+x_1, W_4+x)}\right)$}
\end{minipage}
\begin{minipage}{.29\textwidth}
  \centering
%
\subcaption{$\cI\left(\substack{(W_1+x_1, W_3+x),\\(W_1+x_1, W_2+x_2),\\(W_1+x_1, W_4+x)}\right)$}
\end{minipage}
\begin{minipage}{.29\textwidth}
  \centering
%
\subcaption{$\cI\left(\substack{(W_1+x_1, W_3+x),\\(W_2+x_2, W_4+x),\\(W_1+x_1, W_2+x_2)}\right)$}
\end{minipage}
\caption{Three cases of coalescing of simple random walks.}
\label{fig:coa}
\end{figure}

\begin{proof}
Let $l$ (depending on $\bv$) be the number such that $x \in G_l$.
First observe that when $C(x)\neq \bC(P_x(t))$, one of the following two events happen:
\begin{enumerate}
\item for some $x'\in G_l$, $x'\neq x$, the paths $W_{x,m^*(x)}+x$ and $W_{x',m^*(x')}+x'$ intersect;
\item $l>0$, and $C'_{x,m^*(x)}\neq C(x)$.
\end{enumerate}
This is because, if $l>0$ and $W_{x,m^*(x)}+x$ is disjoint from $W_{x',m^*(x')}+x'$ for any $x'\in G_l$, $x'\neq x$, there must be $\bC(P_x(t)) = C'_{x,m^*(x)}$ from our construction.

We start by considering the probability of the first event.
It can be bounded by
\begin{equation}  \label{eq:ConstrChangeProb:1}
\sum_{x'\in S} \PP[x' \in G_l,\; \exists t'\in[0,t], W_{x',m^*(x')}(t')+x'=W_{x,m^*(x)}(t')+x ].
\end{equation}
By Lemma \ref{lem:CondiRandWalk}, $W_{x',m^*(x')}, W_{x,m^*(x)}$ are independent simple random walks, and are independent of the event $x' \in G_l$.
Let $\overline{W}$ be a rate $2$ simple random walk from time $0$ to $t$. Then $\overline{W}$ has the same law as $W_{x',m^*(x')}- W_{x,m^*(x)}$, and we can bound \eqref{eq:ConstrChangeProb:1} by
\begin{equation}  \label{eq:ConstrChangeProb:2}
M^{-1}\sum_{x' \in \Z^d} \PP[ \exists t'\in[0,t],\overline{W}(t') = x - x' ].
\end{equation}
The summation in \eqref{eq:ConstrChangeProb:2} is precisely the expected number of locations visited by $\overline{W}$ in time $[0, t]$; thus \eqref{eq:ConstrChangeProb:2} is bounded by $2tM^{-1} \leq 2t^{-1}$.

We then consider the probability of the second event.
When $l>0$, denote
\[
b_x := (1-p)\cK_{x,\Xi} + \cK_{x,0}.
\]
From the definition of $m^*(x)$, when $l>0$ we have
\[
\begin{split}
&\PP[C(x)=0, C'_{x,m^*(x)}=1 \mid  \bv, \sP_l, \{\bC(y)\}_{y\in Y_l} ] \\
=&\PP[b_x<w_{x,m^*(x)}<1-p \mid  \bv, \sP_l, \{\bC(y)\}_{y\in Y_l} ] .
\end{split}
\]
Similarly,
\[
\begin{split}
&\PP[C(x)=1, C'_{x,m^*(x)}=0 \mid  \bv, \sP_l, \{\bC(y)\}_{y\in Y_l} ] \\
=&\PP[1-p<w_{x,m^*(x)}<b_x \mid  \bv, \sP_l, \{\bC(y)\}_{y\in Y_l} ] .
\end{split}
\]
Given $\bv, \sP_l, \{\bC(y)\}_{y\in Y_l}$, the distribution of $w_{x,m^*(x)}$ is uniform on $[0,1]$, so
\[
\PP[C(x)\neq C'_{x,m^*(x)} \mid  \bv, \sP_l, \{\bC(y)\}_{y\in Y_l} ] = |b_x-(1-p)|
\]
To get the marginal probability we need to integrate over the conditioned random variables.
We first integrate over the coloring $\bC$. Note that we can write
\begin{equation}  \label{eq:bx}
\begin{split}
(1-p)-b_x
&=
(1-p)(\cK_{x,0}+\cK_{x,\Xi}+\cK_{x,1})
-
\left(
(1-p)\cK_{x,\Xi} + \cK_{x,0}
\right)
\\ &=
(1-p)\cK_{x,1} - p \cK_{x,0}.
\end{split}
\end{equation}
We next write $\cK_{x,1}$ and $\cK_{x,0}$ as the sums of the probabilities that a simple random walk starting from $x$ coalesces into a certain vertex $y\in Y_l$.
More precisely, for any $y \in Y_l$, we denote
\[
\cJ_{x,y} :=
\PP\Big[ J[W + x, \sP_l] (t) = y \mid \bv, \sP_l\Big],
\]
where $W$ is an independent simple random walk.
Then we have
\[
\cK_{x,0} = \sum_{y \in Y_l}\cJ_{x,y}\mathds{1}_{\bC(y)=0}, \quad \cK_{x,1} = \sum_{y \in Y_l}\cJ_{x,y}\mathds{1}_{\bC(y)=1}.
\]
Thus by \eqref{eq:bx}, we have
\[
(1-p)-b_x = (1-p)\cK_{x,1} - p \cK_{x,0} = \sum_{y \in Y_l}\cJ_{x,y}(\mathds{1}_{\bC(y)=1}(1-p)-\mathds{1}_{\bC(y)=0}p). 
\]
This implies that (when $l>0$)
\begin{align*}
&\E[(b_x-(1-p))^2\mid  \bv, \sP_l ]
\\
=&\;
\E\left[\sum_{y,y' \in Y_l}\cJ_{x,y}(\mathds{1}_{\bC(y)=1}(1-p)-\mathds{1}_{\bC(y)=0}p)\cJ_{x,y'}(\mathds{1}_{\bC(y')=1}(1-p)-\mathds{1}_{\bC(y')=0}p)\mid  \bv, \sP_l \right]
\\
=&\;
\sum_{y \in Y_l} \cJ_{x,y}^2((1-p)^2\PP[\bC(y))=1\mid  \bv, \sP_l] + p^2\PP[\bC(y)=0\mid  \bv, \sP_l])
\\
=&\;
p(1-p)\sum_{y \in Y_l} \cJ_{x,y}^2
\leq \frac{1}{4}\sum_{y \in Y_l} \cJ_{x,y}^2,
\end{align*}
where we used that conditioned on $\bv, \sP_l$, the colorings $\bC(y)$ and $\bC(y')$ are independent for any $y\neq y' \in Y_l$, by Lemma \ref{lem:const-prop}.
Thus we have that
\[
\PP[l>0, C(x) \neq C'_{x,m^*(x)}]^2
= 
\E[\mathds{1}_{l>0}|b_x-(1-p)|]^2
\leq 
\E[\mathds{1}_{l>0}(b_x-(1-p))^2]
\leq 
\E\left[\frac{1}{4}\sum_{y \in Y_l} \cJ_{x,y}^2\right].
\]
Next we bound $\E\left[\frac{1}{4}\sum_{y \in Y_l} \cJ_{x,y}^2\right]$.
Conditional on $\bv, \sP_l$ we have
\begin{equation}
\cJ_{x,y} \leq
\sum_{x' \in \hG_l} \mathds{1}_{P_{x'}(t)=y} \PP[\exists t' \in [0,t], P_{x'}(t')=W(t')+x \mid P_{x'} ],
\end{equation}
where $W$ is a simple random walk; and then $\sum_{y \in Y_l} \cJ_{x,y}^2$ can be bounded by
\begin{equation}  \label{eq:ConstrChangeProb:49}
\begin{split}
\sum_{\substack{x_1, x_2 \in S, \\ x_1, x_2 \neq x}}
\mathds{1}_{P_{x_1}(t)=P_{x_2}(t)}
&\PP[\exists t' \in [0,t], P_{x_1}(t')=W(t')+x \mid P_{x_1}]\\ \times &
\PP[\exists t' \in [0,t], P_{x_2}(t')=W(t')+x \mid P_{x_2}].
\end{split}
\end{equation}
For any $x_1, x_2 \in S$, we have that $P_{x_1}, P_{x_2}$ are coalescing simple random walks starting from $x_1, x_2$.
First we consider the case when $x_1=x_2\in S$. 
The summand in \eqref{eq:ConstrChangeProb:49} is the probability that $P_{x_1}$ intersects with $W_2+x$ and $W_3+x$, for two independent simple random walks $W_2, W_3$.
Here $P_{x_1}$ itself is distributed as $W_1+x_1$, for another simple random walk $W_1$.
Thus we have
\[
\begin{split}
&\E[\PP[\exists t' \in [0,t], P_{x_1}(t')=W(t')+x \mid P_{x_1} ]^2 ]
\\
=&
\E[\cI((W_1+x_1, W_2+x),(W_1+x_1, W_3+x)) ] + \E[\cI((W_1+x_1, W_3+x),(W_1+x_1, W_2+x)) ]
\\
=& 2\E[\cI((W_1+x_1, W_2+x),(W_1+x_1, W_3+x)) ],
\end{split}
\]
where the last equality is by symmetry between $W_2, W_3$.
When $x_1\neq x_2$, 
the summand in \eqref{eq:ConstrChangeProb:49} is the probability that each of $W_3+x$ and $W_4+x$ intersects one of $P_{x_1}, P_{x_2}$, for two independent simple random walks $W_3, W_4$.
For $P_{x_1}, P_{x_2}$, they are distributed as $W_1+x_1$, and $J[W_2+x_2, \{W_1+x_1\}]$, for independent simple random walks $W_1, W_2$.

For the summand in \eqref{eq:ConstrChangeProb:49}, we consider three coalescences: between $P_{x_1}$, $P_{x_2}$, between $P_{x_1}$, $W_3+x$, and between $P_{x_2}$, $W_4+x$.
By considering the 6 different orders in which these three coalescences happen, we have
\[
\begin{split} 
&\E\left[\mathds{1}_{P_{x_1}(t)=P_{x_2}(t)} \PP[\exists t' \in [0,t], P_{x_1}(t')=W(t')+x \mid P_{x_1}]\PP[\exists t' \in [0,t], P_{x_2}(t')=W(t')+x \mid P_{x_2}] \right]
\\
=&
\E\left[\cI\left((W_1+x_1, W_2+x_2),(W_1+x_1, W_3+x),(W_1+x_1, W_4+x)\right) \right]
\\
&+\E\left[\cI\left((W_1+x_1, W_2+x_2),(W_1+x_1, W_4+x),(W_1+x_1, W_3+x)\right) \right]
\\
&+\E\left[\cI\left((W_1+x_1, W_3+x),(W_1+x_1, W_2+x_2),(W_1+x_1, W_4+x)\right) \right]
\\
&+\E\left[\cI\left((W_2+x_2, W_4+x),(W_1+x_1, W_2+x_2),(W_1+x_1, W_3+x)\right) \right]
\\
&+\E\left[\cI\left((W_1+x_1, W_3+x),(W_2+x_2, W_4+x),(W_1+x_1, W_2+x_2)\right) \right]
\\
&+\E\left[\cI\left((W_2+x_2, W_4+x),(W_1+x_1, W_3+x),(W_1+x_1, W_2+x_2)\right) \right].
\end{split}
\]
Using the above two equations with \eqref{eq:ConstrChangeProb:49}, we have
\begin{align*}
&\PP[l>0, C(x) \neq C'_{x,m^*(x)}]^2
\\
&\qquad\leq
\sum_{x_1\in S, x_1 \neq x}\frac{1}{2}\E\left[\cI\left((W_1+x_1, W_2+x),(W_1+x_1, W_3+x)\right) \right]
\\
&\qquad
+ \sum_{\substack{x_1,x_2\in S,\\ x\neq x_1, x\neq x_2, x_1\neq x_2}}
\frac{1}{2}\E\left[\cI\left((W_1+x_1, W_2+x_2),(W_1+x_1, W_3+x),(W_1+x_1, W_4+x)\right) \right]
\\
&\qquad
+\frac{1}{2}\E\left[\cI\left((W_1+x_1, W_3+x),(W_1+x_1, W_2+x_2),(W_1+x_1, W_4+x)\right) \right]
\\
&\qquad
+\frac{1}{2}\E\left[\cI\left((W_1+x_1, W_3+x),(W_2+x_2, W_4+x),(W_1+x_1, W_2+x_2)\right) \right],
\end{align*}
where we used symmetry between $x_1$, $x_2$, and symmetry between $W_3$, $W_4$.
The three coalescing events are visualized in Figure \ref{fig:coa}. This establishes \eqref{eq:ConstrChangeProb:0}.
\end{proof}

In order to apply this construction in the next section, we denote it as a measurable and translation invariant function $\cF_t: \cP \times \cQ_t \rightarrow \cR_t$.
Here $\cP$ is the domain of the inputs $S$ and $C$; $\cQ_t$ denotes the domain for the IID randomness; and $\cR_t$ is the domain of the output $\bS, \{P_x\}_{x\in\Z^d}, \bC$.
We can write them as
\[
\cP= \cS \times \cC, \; \cQ_t= ( \cW_t^{\Z_+} \times [0,1]^{\Z_+} \times [0, 1] )^{\Z^d}, \; \cR_t= \cS \times \cW_t^{\Z^d} \times \cC 
\]
where $\cS$ is the collection of all subsets of $\Z^d$, and $\cC:=\{0,1\}^{\Z^d}$ is the space of coloring of $\Z^d$.
The map $\cF_t: \cP \times \cQ_t \to \cR_t$ is defined by our construction
\[
(S, C, \{ (\{W_{x,m}\}_{m=1}^{\infty}, \{u_{x,m}\}_{m=1}^{\infty}, v_x ) \}_{x \in \Z^d} )
\mapsto (\bS, \{P_x\}_{x\in\Z^d}, \bC).
\]
Note that only the coloring of $C$ on $S$ is used, while here we treat the input as a coloring of $\Z^d$. Also, in our construction $\bC$ is defined only on $\bS$, and $P_x$ is only constructed for $x\in S$.
To make this formalism well defined, for each $x \not\in S$, let $P_x\equiv x$; and for each $y \not\in \bS$, let $\bC(y)=0$.
Strictly speaking, this function is actually defined $\cQ_t$-almost surely.

\section{Sequential construction and convergence of measures}   \label{sec:seq-const}
In this section we prove Theorem \ref{thm:main}.
We couple $\cM_{2^k-1}\rho_p$ for all $k=0,1,\ldots$, using the coupling of two times (i.e. the functions $\left\{\cF_{2^{k-1}}\right\}_{k=1}^{\infty}$) given in Section \ref{sec:StepConst}.

The coupling is defined using the following IID process.
For each vertex $x\in \Z^d$ and $k\in\Z_+$, we let $\{W_{x,k,m}\}_{m=1}^{\infty}$ be a sequence of independent simple random walks, each of time  $2^{k-1}$; let $\{u_{x,k,m}\}_{m=1}^{\infty}$ be a sequence of independent uniform $[0, 1]$ random variables, and $v_{x,k}$ be an additional independent uniform $[0, 1]$ random variable.
We also let $C_0:\Z^d\to \{0, 1\}$ be a random coloring such that each $C_0(x)=1$ with probability $p$ independently, and let $S_0=\Z^d$.
For each $k\in \Z_+$, we let
\[
\left(S_k, \{P_{x,k}\}_{x\in\Z^d}, C_k\right) := \cF_{2^{k-1}} \left( S_{k-1}, C_{k-1}, \left\{ \left(\{W_{x,k,m}\}_{m=1}^{\infty}, \{u_{x,k,m}\}_{m=1}^{\infty}, v_{x,k} \right) \right\}_{x \in \Z^d} \right) .
\]
We then define $\hP_{x,k}$ as the concatenation of the collection of paths $\{P_{x,1}\}_{x\in S_0},\ldots, \{P_{x,k}\}_{x\in S_{k-1}}$, and we let $D_k(x):= C_k(\hP_{x,k}(2^k-1))$.
Also denote that $D_0(x):=C_0(x)$.

With this construction and the properties given by Lemma \ref{lem:const-prop}, the paths $\hP_{x,k}$ are distributed as coalescing random walks from time $0$ to $2^{k}-1$; and $S_k$ is the set of walkers of the coalescing random walk at time $2^{k}-1$.
The law of the coloring $D_k(x)$ is given by a product measure with density $p$ over the components of the coalescing paths $\{\hP_{x,k}\}_{x\in\Z^d}$ and so has law $\cM_{2^k-1}\rho_p$.  What remains is to prove that $D_k(x)$ converges almost surely as $k\to \infty$ for each $x\in \Z^d$ which we prove in the following proposition.

\begin{prop}   \label{prop:FiniteChangeTimes}
For each $x \in \Z^d$, almost surely we have
\begin{equation}  \label{eq:prop:FiniteChangeTimes:0}
| \{k \in \Z_{\geq 0}: D_k(x)\neq D_{k+1}(x) \} | < \infty
\end{equation}
\end{prop}
We first prove Theorem \ref{thm:main} assuming Proposition \ref{prop:FiniteChangeTimes}.

\begin{proof}[Proof of Theorem \ref{thm:main}]
By Proposition \ref{prop:FiniteChangeTimes}, almost surely, as $k \rightarrow \infty$, $D_k(x)$ converges for each $x$.
Letting the limit be $D(x)$, then almost surely $D_k \rightarrow D$ in $\{0, 1\}^{\Z^d}$ (in the product topology), and the measure of $D$ must be $\mu_p$, the weak limit of $\cM_{2^{k}-1}\rho_p$ as $k \rightarrow \infty$.
Since for each $k \in \Z_+$, $D_k$ is a measurable and translation invariant function of $(C_0, \left\{ \left(\{W_{x,k,m}\}_{m=1}^{\infty}, \{u_{x,k,m}\}_{m=1}^{\infty}, v_{x,k} \right) \right\}_{x \in \Z^d, k\in \Z_{+}})$, so is $D$.
This means that $D$ is a factor of an IID process on $\Z^d$;
thus $(\{0, 1\}^{\Z^d}, \mu_p)$ with translation is isomorphic to a generalized Bernoulli shift by \cite{ornstein1970factors} and \cite{ornstein1987entropy}.

Finally, for $\{0, 1\}^{\Z^d}$ with translations, the topological entropy is $\log 2$; thus by the variational principle for entropy, the measure theoretical entropy of $\mu_p$ is upper bounded by $\log 2$.
This implies that $(\{0, 1\}^{\Z^d}, \mu_p)$ is isomorphic to a Bernoulli shift (with finite state space) by \cite{ornstein1970bernoulli} and \cite{ornstein1987entropy}.
\end{proof}

To prove Proposition \ref{prop:FiniteChangeTimes}, we need to control the probability that $D_k(x)\neq D_{k+1}(x)$, for each $k\in \Z_{\geq 0}$ and $x \in \Z^d$.
For this we need some basic properties of the set $S_k$ (walkers of coalescing random walks at time $2^{k}-1$) and $\hP_{\boo, k}(2^k-1)$ (the location of the walker starting from $\boo$ at time $2^k-1$).
We note that, while the following lemmas are stated using our construction $\hP_{\boo, k}(2^k-1)$ and $S_k$, they are simply results for coalescing simple random walks.
\begin{lemma}  \protect{\cite[Theorem 1]{bramson1980asymptotics}} \label{lem:ProbInSK}
There exists constant $\lambda \in \R_+$, such that $\PP[\boo \in S_k] < 2^{-k}\lambda$.
\end{lemma}
The next lemma is about negative correlation of locations occupied by walkers, and we postpone its proof to Appendix \ref{sec:a}.
\begin{lemma}  \label{lem:ProbInSKUncor}
For any mutually different $a_1, a_2, a_3 \in \Z^d$, we have that
\begin{equation}  \label{eq:lem:ProbInSKUncor:1}
\PP[a_1=\hP_{\boo, k}(2^k-1), a_2 \in S_k] \leq \PP[a_1=\hP_{\boo, k}(2^k-1)]\PP[\boo \in S_k],
\end{equation}
and
\begin{equation}  \label{eq:lem:ProbInSKUncor:2}
\PP[a_1=\hP_{\boo, k}(2^k-1), a_2, a_3 \in S_k] \leq 2\PP[a_1=\hP_{\boo, k}(2^k-1)]\PP[\boo \in S_k]^2,
\end{equation}
\end{lemma}

Now we finish the proof of Proposition \ref{prop:FiniteChangeTimes}.
The general idea is to apply Proposition \ref{prop:ConstrChangeProb} to each $\PP[D_k(x) \neq D_{k+1}(x)]$,
then use the above lemmas and Proposition \ref{prop:ControlCoalProb1} to show that the bound decays exponentially in $k$.
\begin{proof}[Proof of Proposition \ref{prop:FiniteChangeTimes}]
Without loss of generality we assume that $x = \boo$.
For each $k \in \Z_{\geq 0}$, denote $\nu_k:= \hP_{\boo, k}(2^k-1)$, and we consider $\PP[D_k(\boo) \neq D_{k+1}(\boo)]$.
By Proposition \ref{prop:ConstrChangeProb}, $\PP[D_k(\boo) \neq D_{k+1}(\boo)|S_{k}, \nu_{k}]$
is bounded by
\[
\begin{split}
2^{1-k} + \Bigg(&\sum_{x_1\in S_k, x_1 \neq \nu_k}
\frac{1}{2}\E[\cI((W_{1}+x_1, W_{2}+\nu_k),(W_{1}+x_1, W_{3}+\nu_k)) ]
\\
+ \sum_{\substack{x_1,x_2\in S_k, \\ \nu_k\neq x_1, \nu_k\neq x_2, x_1\neq x_2}}
&\frac{1}{2}\E[\cI((W_{1}+x_1, W_{2}+x_2),(W_{1}+x_1, W_{3}+\nu_k),(W_{1}+x_1, W_{4}+\nu_k)) ]
\\
+&\frac{1}{2}\E[\cI((W_{1}+x_1, W_{3}+\nu_k),(W_{1}+x_1, W_{2}+x_2),(W_{1}+x_1, W_{4}+\nu_k)) ]
\\
+&\frac{1}{2}\E[\cI((W_{1}+x_1, W_{3}+\nu_k),(W_{2}+x_2, W_{4}+\nu_k),(W_{1}+x_1, W_{2}+x_2)) ]\Bigg)^{1/2},
\end{split}
\]
where $W_1, W_2, W_3, W_4$ are independent simple random walks of time $2^k$.
By translation invariance and Cauchy-Schwarz inequality, we can bound the unconditional probability $\PP[D_k(\boo) \neq D_{k+1}(\boo)]$ by
\begin{equation}   \label{eq:prop:FiniteChangeTimes:2}
\begin{split}
2^{1-k} + \Bigg(&\sum_{x_1,x'\in \Z^d, x_1 \neq \boo}\frac{1}{2}
\PP[x_1+x' \in S_k, \nu_k=x']\E[\cI((W_{1}+x_1, W_{2}),(W_{1}+x_1, W_{3})) ]
\\
+ &\sum_{x_1,x_2,x'\in \Z^d, x_1, x_2\neq \boo, x_1\neq x_2} \PP[x_1+x', x_2+x' \in S_k, \nu_k=x']
\\
\times
\Bigg(&\frac{1}{2}
\E[\cI((W_{1}+x_1, W_{2}+x_2),(W_{1}+x_1, W_{3}),(W_{1}+x_1, W_{4})) ]
\\
+&\frac{1}{2}\E[\cI((W_{1}+x_1, W_{3}),(W_{1}+x_1, W_{2}+x_2),(W_{1}+x_1, W_{4})) ]
\\
+&\frac{1}{2}\E[\cI((W_{1}+x_1, W_{3}),(W_{2}+x_2, W_{4}),(W_{1}+x_1, W_{2}+x_2)) ] \Bigg)\Bigg)^{1/2}.
\end{split}
\end{equation}
By Lemma \ref{lem:ProbInSK} and \ref{lem:ProbInSKUncor}, for any $x_1\neq x_2 \in \Z^d$, with $x_1, x_2\neq \boo$, we have
\begin{equation}   \label{eq:prop:FiniteChangeTimes:31}
\PP[x_1+x' \in S_k, \nu_k=x']
\leq
\PP[\nu_k=x']\PP[\boo \in S_k]
\leq
2^{-k}\lambda\PP[\nu_k=x'],
\end{equation}
and
\begin{equation}   \label{eq:prop:FiniteChangeTimes:32}
\PP[x_1+x', x_2+x' \in S_k, \nu_k=x']
\leq
2\PP[\nu_k=x']\PP[\boo \in S_k]^2
\leq 2^{1-2k}\lambda^2\PP[\nu_k=x'],
\end{equation}
where $\lambda$ is the constant in Lemma \ref{lem:ProbInSK}.
By plugging \eqref{eq:prop:FiniteChangeTimes:31} and \eqref{eq:prop:FiniteChangeTimes:32} into \eqref{eq:prop:FiniteChangeTimes:2}, summing over $x'$,
and using Proposition \ref{prop:ControlCoalProb1}, we can bound \eqref{eq:prop:FiniteChangeTimes:2} by
\[
2^{1-k} + \left(\frac{1}{2}\cdot 2^{-k}\lambda \cdot 2^{3/2}\kappa_d 2^{k/2}
+ \frac{1}{2}\cdot 2^{1-2k}\lambda^2\cdot 3 \cdot \frac{2^{3/2}}{3} \kappa_d 2^{3k/2}\right)^{1/2},
\]
where $\kappa_d$ is a constant depending only on dimension $d$, and is defined in Definition \ref{defn:HitProTim}.
Thus we have that
\[
\sum_{k=0}^{\infty} \PP[D_k(\boo) \neq D_{k+1}(\boo)] \leq \sum_{k=0}^{\infty} 2^{-\frac{k}{4}} \left(
2 + \left(\frac{1}{2}\cdot \lambda \cdot 2^{3/2}\kappa_d
+ \frac{1}{2}\cdot 2\lambda^2\cdot 3 \cdot \frac{2^{3/2}}{3} \kappa_d\right)^{1/2}
\right)
< \infty,
\]
and \eqref{eq:prop:FiniteChangeTimes:0} holds almost surely.
\end{proof}

\bibliographystyle{halpha}
\bibliography{bibliography}

\appendix

\section{Remaining walkers in coalescing simple random walks}  \label{sec:a}
This appendix is devoted to prove Lemma \ref{lem:ProbInSKUncor}. The notations used below are from Section \ref{sec:seq-const}.
\begin{proof} [Proof of Lemma \ref{lem:ProbInSKUncor}]
Let $\{x_i\}_{i=1}^{\infty}$ be an iteration of the set $\Z^d\setminus\{\boo\}$.
We take $\{W_i\}_{i=1}^{\infty}, \{W_i'\}_{i=1}^{\infty}$, where each $W_i, W_i'$ are independent simple random walks of time $2^k-1$.

We prove \eqref{eq:lem:ProbInSKUncor:1} first.
It suffices to show that
$\PP[a_2 \in S_k |\hP_{\boo, k}] \leq \PP[a_2 \in S_k]$,
for any path $\hP_{\boo, k}$ with $\hP_{\boo, k}(2^k-1)=a_1$.
Conditioned on $\hP_{\boo, k}$, from its construction the law of $\{\hP_{x_{i},k}\}_{i=1}^{\infty}$ is the same as that of
$J[\{W_j+x_{j}\}_{j=1}^{\infty}, \{\hP_{\boo, k}\}]$.
Also, conditioned on $\hP_{\boo, k}$, we could couple $\{W_i\}_{i=1}^{\infty}, \{W_i'\}_{i=1}^{\infty}$, such that for any $0\leq t \leq 2^k-1$, 
\[
\{J[\{W_j+x_{j}\}_{j=1}^{\infty}, \{\hP_{\boo, k}\}](i, t)\}_{i=1}^{\infty}
\subset
\{J[\{W'_j+x_{j}\}_{j=1}^{\infty}](i, t)\}_{i=1}^{\infty} \cup \{\hP_{\boo, k}(t)\}.
\]
Thus
\[
\begin{split}
\PP[a_2 \in S_k |\hP_{\boo, k}]
=&
\PP[ a_2 \in  \{J[\{W_j+x_{j}\}_{j=1}^{\infty}, \{\hP_{\boo, k}\}](i, 2^k-1)\}_{i=1}^{\infty} \mid \hP_{\boo, k} ]
\\
\leq &
\PP[ a_2 \in  \{J[\{W'_j+x_{j}\}_{j=1}^{\infty}](i, 2^k-1)\}_{i=1}^{\infty} ]\\
= &
\PP[ a_2 \in  \{\hP_{x_{i},k}(2^k-1)\}_{i=1}^{\infty} ] \\
\leq &\PP[a_2 \in S_k]
\end{split}
\]
where we used that
$J[\{W_j'+x_{j}\}_{j=1}^{\infty}]$ has the same law as $\{\hP_{x_{i},k}\}_{i=1}^{\infty}$.

Then we prove \eqref{eq:lem:ProbInSKUncor:2}, using a similar method.
For $\iota=2,3$, denote $i_\iota$ to be the smallest positive integer such that $\hP_{x_{i_\iota},k}(2^k-1) = a_\iota$, and $i_\iota=\infty$ if no such number exists.
It suffices to prove that, for any $j_2 \in \Z_+$,
\begin{equation}  \label{eq:lem:ProbInSKUncor:6}
\PP[a_1=\hP_{\boo, k}(2^k-1), i_2= j_2 < i_3 < \infty] \leq \PP[a_1=\hP_{\boo, k}(2^k-1), i_2=j_2<i_3]\PP[i_3<\infty].
\end{equation}
Then by summing over $j_2 \in \Z_+$, and using symmetry between $a_2, a_3$, we get that
\[
\begin{split}
&\PP[a_1=\hP_{\boo, k}(2^k-1), a_2, a_3 \in S_k] 
\\=&
\PP[a_1=\hP_{\boo, k}(2^k-1), i_2, i_3 < \infty]
\\
\leq & 2\PP[a_1=\hP_{\boo, k}(2^k-1), i_2 < \infty]\PP[i_3<\infty]\\
\leq &
2\PP[a_1=\hP_{\boo, k}(2^k-1), a_2 \in S_k]\PP[a_3 \in S_k],
\end{split}
\]
Then we get \eqref{eq:lem:ProbInSKUncor:2} by applying \eqref{eq:lem:ProbInSKUncor:1} to the right hand side.

It remains to prove \eqref{eq:lem:ProbInSKUncor:6}.
We note that the event $a_1=\hP_{\boo, k}(2^k-1), i_2=j_2<i_3$ is determined by $\hP_{\boo, k}, \{\hP_{x_i,k}\}_{i=1}^{j_2}$, so \eqref{eq:lem:ProbInSKUncor:6} is implied by
\begin{equation}  \label{eq:lem:ProbInSKUncor:62}
\PP[i_3 < \infty \mid \hP_{\boo, k}, \{\hP_{x_i,k}\}_{i=1}^{j_2}] \leq \PP[i_3<\infty],
\end{equation}
for any $\hP_{\boo, k}, \{\hP_{x_i,k}\}_{i=1}^{j_2}$ such that $a_3 \not\in \{\hP_{x_i,k}(2^k-1)\}_{i=1}^{j_2}$ and $\hP_{\boo, k}(2^k-1)=a_1$.

Conditioned on such $\hP_{\boo, k}, \{\hP_{x_i,k}\}_{i=1}^{j_2}$, from its construction, the law of $\{\hP_{x_{i+j_2},k}\}_{i=1}^{\infty}$ is the same as that of
$J[\{W_j+x_{j+j_2}\}_{j=1}^{\infty}, \{\hP_{\boo, k}\}\cup\{\hP_{x_j,k}\}_{j=1}^{j_2}]$.
We could couple $\{W_i\}_{i=1}^{\infty}, \{W_i'\}_{i=1}^{\infty}$ such that for any $0\leq t \leq 2^k-1$,
\[
\begin{split}
&\{J[\{W_j+x_{j+j_2}\}_{j=1}^{\infty}, \{\hP_{\boo, k}\}\cup\{\hP_{x_j,k}\}_{j=1}^{j_2}](i, t)\}_{i=1}^{\infty}
\\
&\subset 
\{J[\{W_j'+x_{j+j_2}\}_{j=1}^{\infty}](i, t)\}_{i=1}^{\infty} \cup \{\hP_{\boo, k}(t)\}\cup\{\hP_{x_i,k}(t)\}_{i=1}^{j_2},
\end{split}
\]
thus as $a_3 \not\in \{\hP_{\boo, k}(2^k-1)\}\cup\{\hP_{x_i,k}(2^k-1)\}_{i=1}^{j_2}$, we have
\[
\begin{split}
&\PP[i_3 < \infty \mid \hP_{\boo, k},\{\hP_{x_i,k}\}_{i=1}^{j_2}]
\\
=&
\PP[ a_3 \in  \{J[\{W_j+x_{j+j_2}\}_{j=1}^{\infty}, \{\hP_{\boo, k}\}\cup\{\hP_{x_j,k}\}_{j=1}^{j_2}](i, 2^k-1)\}_{i=1}^{\infty} \mid \hP_{\boo, k},\{\hP_{x_i,k}\}_{i=1}^{j_2}]
\\
\leq &
\PP[ a_3 \in \{J[\{W_j'+x_{j+j_2}\}_{j=1}^{\infty}](i, 2^k-1)\}_{i=1}^{\infty}].
\end{split}
\]
Now as $J[\{W_j'+x_{j+j_2}\}_{j=1}^{\infty}]$ has the same law as $\{\hP_{x_{i+j_2}}\}_{i=1}^{\infty}$,
the last line of the above equation is bounded by $\PP[a_3 \in \{\hP_{x_{i+j_2}}(2^k-1)\}_{i=1}^{\infty}] \leq \PP[i_3<\infty]$.
Thus we get \eqref{eq:lem:ProbInSKUncor:62}.
\end{proof}

\section{Estimates on coalescence of simple random walks}
In this appendix we prove some bounds about simple random walks, which are used to bound the terms in the right hand side of \eqref{eq:ConstrChangeProb:0}.
We start with some preliminaries.
\begin{defn}   \label{defn:HitProTim}
Take simple random walk $W: \R_{\geq 0} \rightarrow \Z^d$.
We denote the transition probability as $\cT^{t}_x:=\PP[W(t)=x]$, for any $x\in \Z^d$, $t \in \R_+$.

As a classical result, we can take a constant $\kappa_d(>1)$, depending only on the dimension $d$, such that for any $t\in \R_{+}$, and $x \in \Z^d$, we have $\cT^{t}_x < \kappa_d t^{-d/2}$.
\end{defn}

\begin{defn}   \label{defn:RightContVer}
Let $\cB:=\{x \in \Z^d: \|x\|_1 = 1\}$.
For any $t \in \R_+$ and $P \in \cW_t$, let $P^{\vee}$ be its right continuous limit, i.e.
$P^{\vee}(t') = \lim_{\Delta t \downarrow 0} P(t' + \Delta t)$ for any $t' \in [0, t)$, and $P^{\vee}(t)=P(t)$.
\end{defn}

The following proposition is about estimates on meeting probabilities of two or three independent random walks, and is used in the proof of Proposition \ref{prop:FiniteChangeTimes}.
\begin{prop}  \label{prop:ControlCoalProb1}
For any $t\in \R_+$, and independent simple random walks $W_1, W_2, W_3, W_4$ of time $t$, we have
\begin{equation}   \label{eq:prop:ControlCoalProb2:0}
\sum_{x \in \Z^d}
\E\left[\cI\left((W_{1}+x, W_{2}),(W_{1}+x, W_{3})\right) \right] \leq 2^{3/2} \kappa_d t^{1/2},
\end{equation}
\begin{equation}   \label{eq:prop:ControlCoalProb2:10}
\sum_{x_1, x_2\in \Z^d}
\E\left[\cI\left((W_{1}+x_1, W_{2}+x_2),(W_{1}+x_1, W_{3}),(W_{1}+x_1, W_{4})\right) \right]
\leq \frac{2^{3/2}}{3} \kappa_d t^{3/2},
\end{equation}
\begin{equation}   \label{eq:prop:ControlCoalProb2:20}
\sum_{x_1, x_2\in \Z^d}
\E\left[\cI\left((W_{1}+x_1, W_{3}),(W_{1}+x_1, W_{2}+x_2),(W_{1}+x_1, W_{4})\right) \right]
\leq \frac{2^{3/2}}{3} \kappa_d t^{3/2},
\end{equation}
\begin{equation}   \label{eq:prop:ControlCoalProb2:30}
\sum_{x_1, x_2\in \Z^d}
\E\left[\cI\left((W_{1}+x_1, W_{3}),(W_{2}+x_2, W_{4}),(W_{1}+x_1, W_{2}+x_2)\right) \right]
\leq \frac{2^{3/2}}{3} \kappa_d t^{3/2}.
\end{equation}
\end{prop}

\begin{proof}[Proof of \eqref{eq:prop:ControlCoalProb2:0}]
For any $x \in \Z^d$, let $T_{1,x}:\inf\{t': W_1(t')+x=W_2(t')\}\cup \{\infty\}$, and $T_{2,x}:\inf\{t': W_1(t')+x=W_3(t')\}\cup \{\infty\}$.
We need to bound
\[
\sum_{x\in \Z^d}
\int_{0<t_1<t_2<t}\PP[T_{1,x}\in dt_1, T_{2,x} \in dt_2].
\]
For $t_1<t_2$, we have
\begin{equation}   \label{eq:prop:ControlCoalProb1:1}
\begin{split}
&\sum_{x\in \Z^d}\PP[T_{1,x}\in dt_1, T_{2,x} \in dt_2]
\\
=&
\sum_{x,y_1, y_3 \in \Z^d}
\PP[T_{2,x} \in dt_2| \Wv_1(t_1)=y_1,\Wv_2(t_1)=y_1+x,W_3(t_1)=y_3]
\\
&\times
\PP[T_{1,x}\in dt_1, \Wv_1(T_{1,x})=y_1, \Wv_2(T_{1,x})=y_1+x, W_3(T_{1,x})=y_3].
\end{split}
\end{equation}
By the definition of $T_{1,x}$, we have
\begin{multline}   \label{eq:prop:ControlCoalProb1:2}
\PP[T_{1,x}\in dt_1, \Wv_1(T_{1,x})=y_1, \Wv_2(T_{1,x})=y_1+x, W_3(T_{1,x})=y_3]/dt_1
\\
\leq
\lim_{\Delta t \downarrow 0}
(\Delta t)^{-1}\sum_{b \in \cB} \left(\PP[W_1(t_1)=y_1+b, W_2(t_1)=y_1+x, W_1(t_1 + \Delta t)=y_1, W_2(t_1 + \Delta t)=y_1+x]  \right.
\\
\left. +
\PP[W_1(t_1)=y_1, W_2(t_1)=y_1+x+b, W_1(t_1 + \Delta t)=y_1, W_2(t_1 + \Delta t)=y_1+x]
\right)
\PP[W_3(t_1)=y_3]
\\
=
(2d)^{-1}\sum_{b\in \cB}(\cT^{t_1}_{y_1}\cT^{t_1}_{y_1+x+b}+\cT^{t_1}_{y_1+b}\cT^{t_1}_{y_1+x})\cT^{t_1}_{y_3},
\end{multline}
and
\begin{equation}
\begin{split}   \label{eq:prop:ControlCoalProb1:3}
&\PP[T_{2,x} \in dt_2| W_1(t_1)=y_1,W_2(t_1)=y_1+x,W_3(t_1)=y_3]/dt_2
\\
=&
\PP[T_{2,x} \in dt_2| W_3(t_1)-W_1(t_1)=y_3-y_1]/dt_2
\\
\leq &
\lim_{\Delta t \downarrow 0}
(\Delta t)^{-1}\sum_{b \in \cB} \PP[W_3(t_2)-W_1(t_2)=x+b, W_3(t_1 + \Delta t) - W_1(t_1 + \Delta t)=x 
\\
& \mid W_3(t_1)-W_1(t_1)=y_3-y_1 ]
\\
=&
2(2d)^{-1}\sum_{b\in \cB}\cT^{2(t_2-t_1)}_{x+b-y_3+y_1}.
\end{split}
\end{equation}
Plugging \eqref{eq:prop:ControlCoalProb1:2} and \eqref{eq:prop:ControlCoalProb1:3} into \eqref{eq:prop:ControlCoalProb1:1}, we have
\begin{equation}
\begin{split}   \label{eq:prop:ControlCoalProb1:4}
&\sum_{x\in \Z^d}\PP[T_{1,x}\in dt_1, T_{2,x} \in dt_2]/dt_1 dt_2
\\
\leq &
2(2d)^{-2} \sum_{x,y_1, y_3 \in \Z^d, b_1, b_2 \in \cB}
(\cT^{t_1}_{y_1}\cT^{t_1}_{y_1+x+b_1}+\cT^{t_1}_{y_1+b_1}\cT^{t_1}_{y_1+x})\cT^{t_1}_{y_3}
\cT^{2(t_2-t_1)}_{x+b_2-y_3+y_1}.
\end{split}
\end{equation}
We have that, for each $b_1, b_2 \in \cB$,
\[
\begin{split}
&\sum_{x, y_1, y_3 \in \Z^d}
(\cT^{t_1}_{y_1}\cT^{t_1}_{y_1+x+b_1}+\cT^{t_1}_{y_1+b_1}\cT^{t_1}_{y_1+x})\cT^{t_1}_{y_3}
\cT^{2(t_2-t_1)}_{x+b_2-y_3+y_1}
\\
=&
\sum_{x,y_1 \in \Z^d}
(\cT^{t_1}_{y_1}\cT^{t_1}_{y_1+x+b_1}+\cT^{t_1}_{y_1+b_1}\cT^{t_1}_{y_1+x})
\cT^{2t_2-t_1}_{x+b_2+y_1}
\\
=&
\sum_{y_1 \in \Z^d}
\cT^{t_1}_{y_1}\cT^{2t_2}_{b_2-b_1} + \cT^{t_1}_{y_1+b_1}\cT^{2t_2}_{b_2}
=
\cT^{2t_2}_{b_2-b_1} + \cT^{2t_2}_{b_2}
\end{split}
\]
thus \eqref{eq:prop:ControlCoalProb1:4} is bounded by
\[
2(2d)^{-2}\sum_{b_1, b_2 \in \cB}
\cT^{2t_2}_{b_2-b_1} + \cT^{2t_2}_{b_2}
\leq 4 (\kappa_d (2t_2)^{-d/2} \wedge 1)
\leq 4 \kappa_d (2t_2)^{-3/2}
\]
where $\kappa_d$ was defined in Definition \ref{defn:HitProTim}.
With this, we have
\[
\begin{split}
\sum_{x\in \Z^d}
\int_{0<t_1<t_2<t}\PP[T_{1,x}\in dt_1, T_{2,x} \in dt_2]
\leq &
4 \kappa_d
\int_{0<t_1<t_2<t} (2t_2)^{-3/2} dt_1 dt_2
\\
= &
2^{1/2} \kappa_d \int_{0<t_2<t} t_2^{-1/2}dt_2
=
2^{3/2} \kappa_d t^{1/2} ,
\end{split}
\]
and our conclusion follows.
\end{proof}

\begin{proof}[Proof of \eqref{eq:prop:ControlCoalProb2:10}]
We let
\[
\begin{split}
T_{1,x_1,x_2}&:=\inf\{t': W_1(t')+x_1=W_2(t')+x_2\}\cup \{\infty\} \\
T_{2,x_1,x_2}&:=\inf\{t': W_1(t')+x_1=W_3(t')\}\cup \{\infty\} \\
T_{3,x_1,x_2}&:=\inf\{t': W_1(t')+x_1=W_4(t')\}\cup \{\infty\}
\end{split}
\]
for any $x_1, x_2 \in \Z^d$.
We need to bound
\begin{equation}   \label{eq:prop:ControlCoalProb2:2}
\begin{split}
&\sum_{x_1,x_2\in \Z^d}
\int_{0<t_1<t_2<t_3<t}\PP[T_{1,x_1,x_2}\in dt_1, T_{2,x_1,x_2} \in dt_2, T_{3,x_1,x_2} \in dt_3]
\\
=&
\sum_{x_1,x_2,y_1,y_3,y_4,z_1,z_4\in \Z^d}
\int_{0<t_1<t_2<t_3<t}
\PP[T_{3,x_1,x_2} \in dt_3| \Wv_1(t_2)=z_1,W_4(t_2)=z_4]
\\
&\times
\PP[T_{2,x_1,x_2}\in dt_2, \Wv_1(T_{2,x_1,x_2})=z_1, \Wv_3(T_{2,x_1,x_2})=z_1+x_1, W_4(t_2)=z_4
\\
& \quad \mid \Wv_1(t_1)=y_1, W_3(t_1)=y_3, W_4(t_1)=y_4]
\\
&\times
\PP[T_{1,x_1,x_2}\in dt_1, \Wv_1(T_{1,x_1,x_2})=y_1,\Wv_2(T_{1,x_1,x_2})=y_1+x_1-x_2, W_3(t_1)=y_3, W_4(t_1)=y_4]
\end{split}
\end{equation}
By the definition of $T_{1,x_1,x_2}, T_{2,x_1,x_2}, T_{3,x_1,x_2}$, we can bound \eqref{eq:prop:ControlCoalProb2:2} by
\begin{equation}   \label{eq:prop:ControlCoalProb2:3}
\begin{split}
&\sum_{x_1,x_2,y_1,y_3,y_4,z_1,z_4\in \Z^d}
\int_{0<t_1<t_2<t_3<t}
\sum_{b_1,b_2,b_3 \in \cB}
2(2d)^{-1}\cT^{2(t_3-t_2)}_{x_1+z_1-z_4+b_3}
\\
\times &
(2d)^{-1}\left( \cT^{t_2-t_1}_{z_1-y_1}\cT^{t_2-t_1}_{z_1+x_1-y_3+b_2} + \cT^{t_2-t_1}_{z_1-y_1+b_2}\cT^{t_2-t_1}_{z_1+x_1-y_3} \right)
\cT^{t_2-t_1}_{z_4-y_4}
\\
\times &
(2d)^{-1}\left( \cT^{t_1}_{y_1}\cT^{t_1}_{y_1+x_1-x_2+b_1} + \cT^{t_1}_{y_1+b_1}\cT^{t_1}_{y_1+x_1-x_2} \right)
\cT^{t_1}_{y_3}\cT^{t_1}_{y_4}
dt_1 dt_2 dt_3 .
\end{split}
\end{equation}
By summing over $y_4, z_4, y_3, x_2, y_1, x_1, z_1$ sequentially, \eqref{eq:prop:ControlCoalProb2:3} becomes
\[
\begin{split}
&\int_{0<t_1<t_2<t_3<t}
2(2d)^{-3}\sum_{b_1, b_2, b_3 \in \cB}
2\cT^{2t_3}_{b_2-b_3} + 2\cT^{2t_3}_{-b_3} dt_1 dt_2 dt_3
\\
\leq &\int_{0<t_1<t_2<t_3<t} 8 (\kappa_d(2t_3)^{-d/2} \wedge 1) dt_1 dt_2 dt_3
\\
\leq &\int_{0<t_1<t_2<t_3<t} 8 \kappa_d (2t_3)^{-3/2} dt_1 dt_2 dt_3
\\
= &
\int_{0<t_3<t}
2^{1/2} \kappa_d t_3^{1/2} dt_3
=
\frac{2}{3}\cdot 2^{1/2} \kappa_d t^{3/2}
\end{split}
\]
and \eqref{eq:prop:ControlCoalProb2:10} follows.
\end{proof}

\begin{proof}[Proof of \eqref{eq:prop:ControlCoalProb2:20}]
For $x_1, x_2 \in \Z^d$ we let
\[
\begin{split}
T_{1,x_1,x_2}&:=\inf\{t': W_1(t')+x_1=W_3(t')\}\cup \{\infty\} \\
T_{2,x_1,x_2}&:=\inf\{t': W_1(t')+x_1=W_2(t')+x_2\}\cup \{\infty\} \\
T_{3,x_1,x_2}&:=\inf\{t': W_1(t')+x_1=W_4(t')\}\cup \{\infty\} .
\end{split}
\]
As in the proof of \eqref{eq:prop:ControlCoalProb2:10}, we just need to bound
\begin{equation}   \label{eq:prop:ControlCoalProb2:22}
\begin{split}
&\sum_{x_1,x_2\in \Z^d}
\int_{0<t_1<t_2<t_3<t}\PP[T_{1,x_1,x_2}\in dt_1, T_{2,x_1,x_2} \in dt_2, T_{3,x_1,x_2} \in dt_3]
\\
=&
\sum_{x_1,x_2,y_1,y_2,y_4,z_1,z_4\in \Z^d}
\int_{0<t_1<t_2<t_3<t}
\PP[T_{3,x_1,x_2} \in dt_3| \Wv_1(t_2)=z_1,W_4(t_2)=z_4]
\\
&\times
\PP[T_{2,x_1,x_2}\in dt_2, \Wv_1(T_{2,x_1,x_2})=z_1, \Wv_2(T_{2,x_1,x_2})=z_1+x_1-x_2, W_4(t_2)=z_4
\\
&\quad \mid \Wv_1(t_1)=y_1,W_2(t_1)=y_2, W_4(t_1)=y_4]
\\
&\times
\PP[T_{1,x_1,x_2}\in dt_1, \Wv_1(T_{1,x_1,x_2})=y_1,\Wv_3(T_{1,x_1,x_2})=y_1+x_1, W_2(t_1)=y_2, W_4(t_1)=y_4]
\\
\leq &
\sum_{x_1,x_2,y_1,y_2,y_4,z_1,z_4\in \Z^d}
\int_{0<t_1<t_2<t_3<t}
\sum_{b_1,b_2,b_3 \in \cB}
2(2d)^{-1}\cT^{2(t_3-t_2)}_{x_1+z_1-z_4+b_3}
\\
&\times
(2d)^{-1}\left( \cT^{t_2-t_1}_{z_1-y_1}\cT^{t_2-t_1}_{z_1+x_1-x_2-y_2+b_2} + \cT^{t_2-t_1}_{z_1-y_1+b_2}\cT^{t_2-t_1}_{z_1+x_1-x_2-y_2} \right)
\cT^{t_2-t_1}_{z_4-y_4}
\\
&\times
(2d)^{-1}\left( \cT^{t_1}_{y_1}\cT^{t_1}_{y_1+x_1+b_1} + \cT^{t_1}_{y_1+b_1}\cT^{t_1}_{y_1+x_1} \right)
\cT^{t_1}_{y_2}\cT^{t_1}_{y_4}
dt_1 dt_2 dt_3 .
\end{split}
\end{equation}
By summing over $y_4, z_4, x_2, y_2, z_1, x_1, y_1$ sequentially, \eqref{eq:prop:ControlCoalProb2:22} becomes
\[
\begin{split}
&\int_{0<t_1<t_2<t_3<t}
2(2d)^{-3}\sum_{b_1, b_2, b_3 \in \cB}
\cT^{2t_3}_{-b_3}+ \cT^{2t_3}_{b_1-b_3}+\cT^{2t_3}_{b_2-b_3}+\cT^{2t_3}_{b_1+b_2-b_3}dt_1 dt_2 dt_3
\\
\leq &\int_{0<t_1<t_2<t_3<t} 8  (\kappa_d(2t_3)^{-d/2} \wedge 1) dt_1 dt_2 dt_3
\leq
\frac{2}{3}\cdot 2^{1/2} \kappa_d t^{3/2}
\end{split}
\]
and \eqref{eq:prop:ControlCoalProb2:20} follows.
\end{proof}

\begin{proof}[Proof of \eqref{eq:prop:ControlCoalProb2:30}]
Again, for $x_1, x_2 \in \Z^d$, let
\[
\begin{split}
T_{1,x_1,x_2}&:=\inf\{t': W_1(t')+x_1=W_3(t')\}\cup \{\infty\} \\
T_{2,x_1,x_2}&:=\inf\{t': W_2(t')+x_2=W_4(t')\}\cup \{\infty\} \\
T_{3,x_1,x_2}&:=\inf\{t': W_1(t')+x_1=W_2(t')+x_2\}\cup \{\infty\}
\end{split}
\]
As in the proof of \eqref{eq:prop:ControlCoalProb2:10}, we just need to bound
\begin{equation}   \label{eq:prop:ControlCoalProb2:32}
\begin{split}
&\sum_{x_1,x_2\in \Z^d}
\int_{0<t_1<t_2<t_3<t}\PP[T_{1,x_1,x_2}\in dt_1, T_{2,x_1,x_2} \in dt_2, T_{3,x_1,x_2} \in dt_3]
\\
=&
\sum_{x_1,x_2,y_1,y_2,y_4,z_1,z_2\in \Z^d}
\int_{0<t_1<t_2<t_3<t}
\PP[T_{3,x_1,x_2} \in dt_3| W_1(t_2)=z_1, \Wv_2(t_2)=z_2]
\\
&\times
\PP[T_{2,x_1,x_2}\in dt_2, W_1(t_2)=z_1, \Wv_2(T_{2,x_1,x_2})=z_2, \Wv_4(T_{2,x_1,x_2})=z_2+x_2
\\
&\quad \mid \Wv_1(t_1)=y_1,W_2(t_1)=y_2, W_4(t_1)=y_4]
\\
&\times
\PP[T_{1,x_1,x_2}\in dt_1, \Wv_1(T_{1,x_1,x_2})=y_1,\Wv_3(T_{1,x_1,x_2})=y_1+x_1, W_2(t_1)=y_2, W_4(t_1)=y_4]
\\
\leq &
\sum_{x_1,x_2,y_1,y_2,y_4,z_1,z_2\in \Z^d}
\int_{0<t_1<t_2<t_3<t}
\sum_{b_1,b_2,b_3 \in \cB}
2(2d)^{-1}\cT^{2(t_3-t_2)}_{x_1-x_2-z_2+z_1+b_3}
\\
&\times
(2d)^{-1}\left( \cT^{t_2-t_1}_{z_2-y_2}\cT^{t_2-t_1}_{z_2+x_2-y_4+b_2} + \cT^{t_2-t_1}_{z_2-y_2+b_2}\cT^{t_2-t_1}_{z_2+x_2-y_4} \right)
\cT^{t_2-t_1}_{z_1-y_1}
\\
&\times
(2d)^{-1}\left( \cT^{t_1}_{y_1}\cT^{t_1}_{y_1+x_1+b_1} + \cT^{t_1}_{y_1+b_1}\cT^{t_1}_{y_1+x_1} \right)
\cT^{t_1}_{y_2}\cT^{t_1}_{y_4}
dt_1 dt_2 dt_3 .
\end{split}
\end{equation}
By summing over $z_1, y_2, y_4, x_1, x_2, z_2, y_1$ sequentially, \eqref{eq:prop:ControlCoalProb2:32} becomes
\[
\begin{split}
&\int_{0<t_1<t_2<t_3<t}
2(2d)^{-3}\sum_{b_1, b_2, b_3 \in \cB}
\cT^{2t_3}_{b_3}+ \cT^{2t_3}_{-b_1+b_3}+\cT^{2t_3}_{b_2+b_3}+\cT^{2t_3}_{-b_1+b_2+b_3}dt_1 dt_2 dt_3
\\
\leq &\int_{0<t_1<t_2<t_3<t} 8  (\kappa_d(2t_3)^{-d/2} \wedge 1) dt_1 dt_2 dt_3
\leq
\frac{2}{3}\cdot 2^{1/2} \kappa_d t^{3/2}
\end{split}
\]
and \eqref{eq:prop:ControlCoalProb2:30} follows.
\end{proof}

\end{document}